\documentclass[12pt]{article}
\usepackage{}
\usepackage[numbers,sort&compress]{natbib}

\usepackage{color}

\definecolor{red}{rgb}{1,0,0}

\usepackage{t1enc}
\usepackage{mathrsfs}
\usepackage[latin1]{inputenc}
\usepackage[english]{babel}
\usepackage{amsmath,amsthm}
\usepackage{amsfonts}
\usepackage{latexsym}
\usepackage{graphicx}
\usepackage{txfonts}
\usepackage[natural]{xcolor}
\usepackage{lineno}
\usepackage{booktabs}

\newtheorem{theorem}{Theorem}
\newtheorem{corollary}[theorem]{Corollary}

\newtheorem{lemma}[theorem]{Lemma}
\newtheorem{remark}[theorem]{Remark}

\newtheorem{conjecture}[theorem]{Conjecture}

\newtheorem{example}[theorem]{Example}
\newtheorem{definition}[theorem]{Definition}

\def \Spec{\textup{Spec}}
\def \T{\textup{T}}
\def \rank{\textup{rank}}
\def \diag{\textup{diag~}}
\def \Im{\textup{Im}}
\def \Span{\textup{Span~}}
\def \Tr{\textup{Tr~}}
\def \Re{\textup{Re}}
\def \lcm{\textup{lcm}}
\title{Generalized spectral characterization of mixed graphs\thanks{Supported by the
National Natural Science Foundation of China under Grant No.\,11971376, 11561058 and 11971406.}}
\author{\small Wei Wang$^{{\rm a}}$\quad Lihong Qiu$^{\rm a}$\quad Jianguo Qian$^{\rm b}$ \quad Wei Wang$^{\rm a,c}$\thanks{Corresponding author: wangwei.math@gmail.com.}
\\
{\footnotesize$^{\rm a}$School of Mathematics and Statistics, Xi'an Jiaotong University, Xi'an 710049,P. R. China}\\
{\footnotesize$^{\rm b}$School of Mathematical Sciences, Xiamen University, Xiamen 361005, P. R. China}\\
{\footnotesize$^{\rm c}$College of Information Engineering, Anhui polytechnic University, Wuhu 241000, P. R. China}
}
\date{}
\begin{document}
 \maketitle
\begin{abstract}
A mixed graph $G$ is a graph obtained from a simple undirected graph by orientating a subset of edges.  $G$ is \emph{self-converse} if it is isomorphic to the graph obtained from $G$ by reversing each directed edge.  For two mixed graphs $G$ and $H$ with Hermitian adjacency matrices $A(G)$ and $A(H)$, we say $G$ is $\mathbb{R}$\emph{-cospectral} to $H$ if, for any $y\in \mathbb{R}$, $yJ-A(G)$ and $yJ-A(H)$ have the same spectrum, where $J$ is the all-one matrix. A self-converse mixed graph $G$ is said to be \emph{determined by its generalized spectrum}, if any self-converse mixed graph
that  is $R$-cospectral with $G$ is isomorphic to $G$. Let $G$ be a self-converse mixed graph of order $n$ such that $2^{-\lfloor n/2\rfloor}\det W$ (which is always a real or pure imaginary Gaussian integer) is  square-free in $\mathbb{Z}[i]$, where $W=[e,Ae,\ldots,A^{n-1}e]$, $A=A(G)$ and $e$ is the all-one vector. We prove that, for any  self-converse mixed graph $H$ that is $\mathbb{R}$-cospectral to $G$, there exists a Gaussian rational unitary matrix $U$ such that $Ue=e$, $U^*A(G)U=A(H)$ and $(1+i)U$ is a Gaussian integral matrix. In particular, if $G$ is an ordinary graph (viewed as a mixed graph) satisfying the above condition, then any self-converse mixed graph $H$ that is $\mathbb{R}$-cospectral to $G$ is $G$ itself (in the sense of isomorphism). This strengthens a recent result of the first author.\\

\end{abstract}
\section{Introduction}
\label{intro}
 Let $G$ be a simple graph with $(0,1)$-adjacency matrix $A(G)$. The \emph{spectrum} of $G$, denoted by $\Spec(G)$, is the multiset of the eigenvalues of $A(G)$. Two graphs $G$ and $H$ are \emph{cospectral} if $\Spec(G) =\Spec (H)$. Trivially, isomorphic graphs are cospectral.  However, the converse is not true in general. A graph $G$ is said to be \emph{determined by its spectrum} (DS for short) if any graph cospectral to $G$ is isomorphic to $G$. It is a fundamental and challenging problem to characterize which graphs are DS.  Although it was conjectured that almost all graphs are DS \cite{haemers}, it is usually extremely difficult to prove a given graph to be DS. For basic results on spectral characterizations of graphs, we refer the readers to the survey papers \cite{ervdamLAA2003,ervdamDM2009}.

In recent years, Wang and Xu \cite{wang2006LAA,wang2006EuJC} and Wang \cite{wang2013ElJC,wang2017JCTB} considered a variant of the above problem. For a graph $G$, the \emph{generalized  spectrum} is the ordered pair $(\Spec (G),\Spec(\overline{G}))$, where $\overline{G}$ denotes the complement of $G$. A graph $G$ is said to be \emph{determined by its generalized spectrum} (DGS for short) if any graph having the same generalized spectrum with $G$ is isomorphic to $G$. For $y\in\mathbb{R}$, two graphs $G$ and $H$ are \emph{$y$-cospectral}  if $yJ-A(G)$ and $yJ-A(H)$ have the same spectrum. Moreover, we say that $G$ and $H$ are \emph{$\mathbb{R}$-cospectral} if $G$ and $H$ are $y$-cospectral for any $y\in \mathbb{R}$.  A classical result of Johnson and Newman \cite{johnson1980JCTB} says that if two graphs are $y$-cospectral for two distinct values of $y$ then they are for all $y$. Therefore, if two graphs $G$ and $H$ are cospectral with cospectral complement, i.e., $G$ and $H$ are $0$-cospectral and $1$-cospectral, then they are $\mathbb{R}$-cospectral.

Let $G$ be a graph with $n$ vertices, $A=A(G)$ and $e$ be the all-one vector of dimension $n$. Let $W(G)=[e,Ae,\ldots,A^{n-1}e]$  be its \emph{walk-matrix}. The following simple arithmetic criterion for graphs being DGS was conjectured in \cite{wang2013ElJC} and finally proved in \cite{wang2017JCTB}.
\begin{theorem}\cite{wang2013ElJC,wang2017JCTB}\label{mainforgraph}
Let $G$ be a graph with $n$ vertices. If $\frac{\det W(G)}{2^{\lfloor n/2 \rfloor}}$ (which is always an integer) is odd and square-free, then $G$ is DGS.
\end{theorem}

Similar result was established for generalized $Q$-spectrum in \cite{qiutoappearDM}. Moreover, Qiu \emph{et al}.~\cite{qiu2019EJC} also gave an analogue of Theorem~\ref{mainforgraph} for Eulerian graphs.  We try to extend Theorem~\ref{mainforgraph} from ordinary graphs to self-converse mixed graphs. A \emph{mixed} graph $G$ is obtained from a simple undirected graph by orientating a subset of edges. For a mixed graph $G$, the \emph{converse} of $G$, denoted by $G^\T$, is the mixed graph obtained from $G$ by reversing each directed edge in $G$.  A mixed graph is said to be \emph{self-converse} if $G^\T$ is isomorphic to $G$.  As a trivial example, each simple undirected graph is self-converse as $G^\T=G$ in this case. For a mixed graph $G$, we use the symbol $u\sim v$ to denote that $uv$ is an undirected edge, and use $u\rightarrow v$ (or $v\leftarrow u$) to denote that $uv$ is a directed edge from $u$ to $v$.  The following definition introduced independently by Liu and Li \cite{liu2015LAA} as well as Guo and Mohar~\cite{guo2017JGT} is a natural generalization of adjacency matrix from ordinary graphs to mixed graphs. We use $\mathcal{G}_n$ to denote the set of all mixed graphs with vertex set $V=[n]=\{1,2,\ldots,n\}$. The subset of all self-converse mixed graphs in  $\mathcal{G}_n$ will be denoted by $\mathcal{G}_n^{sc}$.
\begin{definition} \cite{liu2015LAA,guo2017JGT} Let $G\in \mathcal{G}_n$.  The Hermitian adjacency matrix of $G$ is the matrix $A=(a_{u,v})\in {\mathbb C}^{n\times n}$, where
\begin{equation}
a_{u,v}= \begin{cases}
   1  &\mbox{if $u\sim v$,}\\
   i  &\mbox{if $u\rightarrow v$,}\\
   -i &\mbox{if $u\leftarrow v$,}\\
   0  &\mbox{otherwise.}
   \end{cases}
\end{equation}
\end{definition}

Note that for any mixed graph $G$, $A(G)$ is a Hermitian matrix, that is, $A(G)^*=A(G)$, where $A(G)^*$ denotes the conjugate transpose of $A(G)$. Therefore, all eigenvalues of $A(G)$ are real and $A(G)$ is diagonalizable. Also note that $A(G^\T)$ equals $(A(G))^\T$, the transpose of $A(G)$, and this explains why we use $G^\T$ to denote the converse of $G$.  For a mixed graph $G$, the \emph{(Hermitian) spectrum} of $G$, denoted by $\Spec(G)$, is the multiset of the eigenvalues of $A(G)$.
It was observed in \cite{guo2017JGT} that any mixed graph $G$ is cospectral to its converse $G^\T$ since $A(G^\T)=(A(G))^\T$. Indeed, for any $y\in \mathbb{R}$, $yJ-A(G^\T)$ and $yJ-A(G)$ have the same spectrum.

Besides the operation of reversing all directed edges, Guo and Mohar \cite{guo2017JGT} found another important operation, called  four-way switching, which also preserves the Hermitian spectrum. It turns out that extremely rare mixed graphs are determined by their Hermitian spectra \cite{guo2017JGT,wissing}. Indeed, there are 1,540,944 unlabeled mixed graphs of order 6, only 16 of them are determined by their Hermitian spectra \cite[Table 1]{guo2017JGT}.  In \cite{mohar},  Mohar considered spectral determination of classes of switching equivalent mixed graphs, rather than individual graphs. A mixed graph $G$ is determined by its Hermitian spectrum (in the sense of Mohar \cite{mohar}) if every mixed graph with the same Hermitian spectrum can be obtained from $G$ by a four-way switching, possibly followed by the reversal of all directed edges; see \cite{akbaria} for more results along this line.

In this paper, we consider spectral determination of mixed graphs in the sense of generalized spectra, where the generalized spectrum of $G$ means the ordered pair  $(\Spec(G),\Spec (J-I-A(G))$. We found that although the four-way switching operation preserves the spectrum, it usually changes its generalized spectrum. Due to the aforementioned fact that a non-self-converse mixed graph cannot be determined by any kinds of spectra, it is reasonable to restrict ourselves to self-converse mixed graphs. The following definition is a natural generalization of the DGS problem from ordinary graphs to self-converse mixed graphs.
\begin{definition}\label{defDGS}
\textup{ A mixed graph $G\in \mathcal{G}_n^{sc}$ is said to be \emph{determined by generalized (Hermitian) spectrum} if for any  $H\in \mathcal{G}_n^{sc}$,
\begin{equation}\label{sgs}
(\Spec(H),\Spec({J-I-A(H)}))=(\Spec(G),\Spec({J-I-A(G)}))
 \end{equation}
implies that $H$ is isomorphic to $G$.}
\end{definition}
\begin{remark}
For an ordinary graph $G$, $J-I-A(G)$ is the adjacency matrix of the complement $\overline{G}$. Such an explanation is not available for mixed graphs.
\end{remark}
\begin{remark}
As we shall see later, the assumption that $H\in \mathcal{G}_n^{sc}$ is essential to our discussion. Of course, it is natural and desirable to consider the corresponding concept without the assumption that $H$ is self-converse; see \cite{wissing}.
\end{remark}
For $G\in \mathcal{G}_n$, we also define $W(G)=[e,Ae,A^2e,\ldots,A^{n-1}e]$ and call it the walk-matrix of $G$. As $W(G)$ has complex entries, the determinant of $W(G)$ is usually not real. The following simple result illustrates an important property on $W(G)$ when $G$ is self-converse.
\begin{theorem}\label{Wrealorimag}
Let $G\in \mathcal{G}_n^{sc}$. Then there exists a permutation matrix $P$ such that $\overline{W(G)}=P^{-1}W(G)$. In particular, $\det W(G)$ is real or pure imaginary.
\end{theorem}
\begin{proof}
As $A(G^\T)=(A(G))^\T=\overline{A(G)}$ and $\overline{e}=e$, we have
\begin{equation}
W(G^\T)=\left[e,A(G^\T)e,\ldots,A^{n-1}(G^\T)e\right]=\left[\overline{e},\overline{A(G)e},\ldots,\overline{A^{n-1}(G)e}\right]=\overline{W(G)}.
\end{equation}
On the other hand, as $G$ is self-converse, there exists a permutation matrix $P$ such that $A(G^\T)=P^{-1} A(G)P$. As $Pe=e$, we have
$A^k(G^\T)e=P^{-1}A^k(G)Pe=P^{-1}A^k(G)e$ and hence
\begin{equation}
W(G^\T)=\left[P^{-1}e,P^{-1}A(G)e,\ldots,P^{-1}A^{n-1}(G)e\right]=P^{-1}W(G).
\end{equation}
Thus $\overline{W(G)}=P^{-1}W(G)$. Taking determinants on both sides and noting that $\det P^{-1}=\pm 1$, we obtain
$\det \overline{W(G)}=\pm \det W(G)$. Therefore, $\overline{ \det W(G)}=\pm \det W(G)$, which implies that $\det W(G)$ is real or pure imaginary.
\end{proof}
It is known that $2^{\lfloor n/2 \rfloor}$ always divides $\det W(G)$ for ordinary graphs $G$. We shall show that this fact can be extended to mixed graphs in the sense of Gaussian integers.
We believe that the following generalization of Theorem \ref{mainforgraph} is true.
\begin{conjecture}\label{mainconj}
Let $G\in \mathcal{G}_n^{sc}$. If $\frac{|\det W(G)|}{2^{\lfloor n/2 \rfloor}}$ is odd and square-free, then $G$ is DGS.
 \end{conjecture}
The following example shows that Conjecture \ref{mainconj} would be \emph{false} if we remove the restriction that $H\in \mathcal{G}_n^{sc}$ in Definition \ref{defDGS}.
\begin{example}
 Let $G$ and $H$ be two mixed graphs as shown in Fig.~1. Note that $G$ is self-converse but $H$ is not. Direct calculations show that $\det W(G)=-68=-2^2\times 17$,
 $\det (\lambda I-A(G))= \det (\lambda I-A(H))=\lambda^5-7\lambda^3-4\lambda^2+7\lambda+4$ and
 $\det (\lambda I-(J-I-A(G)))= \det (\lambda I-(J-I-A(H)))=\lambda^5-13\lambda^3-16\lambda^2+5\lambda+4$.
 Thus, $G$ and $H$ are $\mathbb{R}$-cospectral.
 \end{example}
 \begin{figure}[htbp]
\begin{center}
\includegraphics[height=4.2cm]{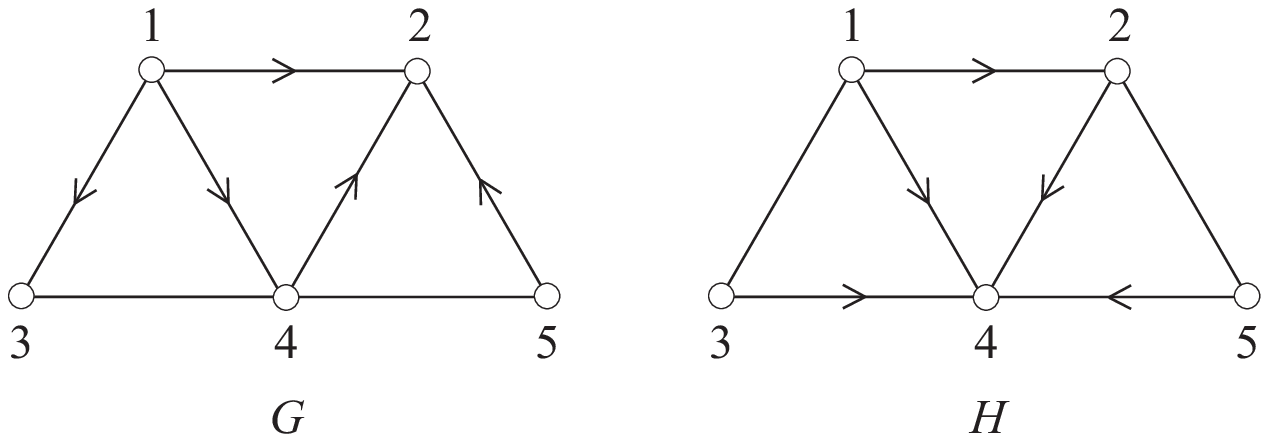}

{\bf Fig.~1~}{Two $\mathbb{R}$-cospectral but not isomorphic graphs}
\end{center}
\label{nK4}
\end{figure}

The main aim of this paper is to give some evidences to support Conjecture  \ref{mainconj}. We have verified the conjecture  for $n\le 6$.  The computer results are given in Table \ref{computer}. The second column gives the number of isomorphic class of $\mathcal{G}_n^{sc}$.  Note that the determinants of all walk matrices in an isomorphic class of $\mathcal{G}_n^{sc}$ have the same absolute value. The third column gives the fractions of DGS graphs in $\mathcal{G}_n^{sc}$, while the fourth column gives the fractions of graphs satisfying the condition of Conjecture \ref{mainconj}. Our experiment shows that, for $n\le 6$, each graph satisfying the condition of Conjecture \ref{mainconj} is DGS.  Although quite a lot of graphs in $\mathcal{G}_n^{sc}$ are DGS,  only a small fractions of them satisfy the condition of Conjecture \ref{mainconj}.  In other words, even Conjecture \ref{mainconj} is true, the condition seems far from necessary.

\begin{table}[htbp]
 \centering
 \caption{\label{computer} Fractions of DGS self-converse mixed graphs}
 \begin{tabular}{rrrrr}
  \toprule
$n$ & Isomorphic Class of $\mathcal{G}_n^{sc}$  & DGS & Conjecture \ref{mainconj}\\
  \midrule
2 &  3 & 1.000&0.333\\
3 & 10 & 1.000&0.100\\
4 & 70& 0.914 &0.086\\
5 & 708& 0.852&0.076\\
6 & 15224  &0.832&0.054\\
  \bottomrule
 \end{tabular}
\end{table}

The rest of the paper is organized as follows. In Section 2, we review some basic facts about Gaussian integers and Gaussian rational unitary matrix.
In Section 3,  we give some  divisibility relations that will be needed later in the paper. In Section 4, we present the main result of the paper together with its proof, which strongly supports our main conjecture above. In Section 5, we verify this conjecture for the special case when $G$ is undirected (and hence trivially self-converse). Conclusions and future work are given in Section 6.
\section{ Gaussian rational unitary matrix and its level}
We recall some facts about Gaussian integers.

The Gaussian integers are the elements of the set $\mathbb{Z}[i]=\{a+bi\colon\,a,b\in \mathbb{Z}\}$, where $i=\sqrt{-1}$. For a Gaussian integer $z=a+bi$, the \emph{norm} of $z$ is $N(z)=a^2+b^2$. Note that $N(z_1z_2)=N(z_1)N(z_2)$. The units of $\mathbb{Z}[i]$ are four powers of $i$, that is, $i,-1,-i,1$. Two Gaussian integers $z_1$ and $z_2$ are \emph{associates} if $b=ua$ for some unit $u$.  It is well known that $\mathbb{Z}[i]$ is a Euclidean domain and hence a unique factorization domain. A nonzero Gaussian integer $z$ is a \emph{Gaussian prime} if it is not a unit and is
divisible only by units and its associates. A positive prime in $\mathbb{Z}$ is a Gaussian prime if and only if $p\equiv 3\pmod 4$. If $p$ is positive prime in $\mathbb{Z}$ such that $p\equiv 1\pmod{4}$ then $p$ can be factored uniquely to a product of two conjugate Gaussian primes (up to multiplication by units and the order of the factors). For example, $5=(1+2i)(1-2i)=(2+i)(2-i)$. In addition, $2$ is not a Gaussian prime, as $2=(1+i)(1-i)$.

We call a Gaussian integer $z$ \emph{even} (resp. \emph{odd}) if $\Re(z)-\Im(z)\equiv 0\pmod{2}$ (resp. $\Re(z)-\Im(z)\equiv 1\pmod{2}$. We call $z\in \mathbb{Z}[i]$ \emph{square-free} if $p^2\not\mid z$ for any Gaussian prime $p$. In particular, 2 is not square-free in $\mathbb{Z}[i]$, but any ordinary odd prime is square-free in $\mathbb{Z}[i]$.  We note that an integer $z\in \mathbb{Z}$ is square-free in $\mathbb{Z}[i]$  if and only if  $z$ is \emph{odd} and square-free (in the ordinary sense).

For a Gaussian prime $p=a+bi$, the quotient ring $\mathbb{Z}[i]/(p)=GF(N(p))$, where $GF(N(p))$ is the field of order $N(p)=a^2+b^2$. As a simple example, $1+i$ is a Gaussian prime and $\mathbb{Z}[i]/(1+i)$ is the binary field $GF(2)$.

A \emph{Gaussian rational} is a complex number whose real part and imaginary part are rational. A \emph{unitary matrix} is a matrix $U\in\mathbb{C}^{n\times n}$ satisfying $U^*U=I$. The following result is a natural generalization of a result for adjacency matrix of an undirected graph obtained in \cite{johnson1980JCTB, wang2006EuJC}. The proof is omitted here since the previous proof is also valid by some slight and evident modification.
\begin{theorem}\label{reduce2unitary}
Let $G\in \mathcal{G}_n$. There exists $H$ such that $G$ and $H$ are cospectral with respect to the generalized spectrum
if and only if there exists a  unitary matrix $U$ satisfying
\begin{equation}\label{gru}
U^*A(G)U=A(H),Ue=e.
\end{equation}
Moreover, if $\det W(G)\neq 0$ then $U=W(G)W^{-1}(H)$ and hence is unique and Gaussian rational.
\end{theorem}

Let $G,H\in\mathcal{G}_n^{sc}$. Define
$$\mathscr{U}_G(H)=\{U\textup{~is Gaussian rational unitary}\colon\,U^*A(G)U=A(H)\text{~and~} Ue=e\},$$
and $\mathscr{U}_G= \cup\mathscr{U}_G(H)$, where the union is taken over all $H\in\mathcal{G}_n^{sc}$.

Under the assumption that $\det W(G)\neq 0$, the structure of $\mathscr{U}_G(H)$ is simple. It is either a singleton or an empty set  depending on whether (\ref{sgs}) holds or not. Furthermore, if (\ref{sgs}) holds, then $\mathscr{U}_G(H)=\{W(G)W^{-1}(H)\}$. In addition, if $G$ and $H$ are isomorphic, i.e., there exists a permutation matrix $P$ with $P^*A(G)P=A(H)$, then $\mathscr{U}_G(H)=\{P\}$  as $P$ is clearly Gaussian rational unitary and $Pe=e$. On the other hand, if (\ref{sgs}) holds but $H$ is not isomorphic to $G$, then the unique element in $\mathscr{U}_G(H)$ is not a permutation matrix.

 Therefore, if $G$ is DGS, then either $\mathscr{U}_G(H)=\emptyset$ or  $\mathscr{U}_G(H)$ consists of a single permutation matrix. Thus, $\mathscr{U}_G$ contains only permutation matrices. If $G$ is not DGS, then there exists $H$ such that (\ref{sgs}) holds but $H$ is not isomorphic to $G$. For such an $H$, the matrix in $\mathscr{U}_G(H)$ is not a permutation matrix and hence $\mathscr{U}_G$ contains matrices other than permutation matrices. We summarize this as the following theorem, which was observed in \cite{wang2006EuJC} for ordinary graphs.
 \begin{theorem}\label{dgsper}
 Let $G\in\mathcal{G}_n^{sc}$ such that $\det W(G)\neq 0$. Then $G$ is DGS if and only if $\mathscr{U}_G$ contains only permutation matrices.
 \end{theorem}

 Let
 \begin{equation}
 \Gamma=\{z\in \mathbb{Z}[i]\colon\,\textup{Re}(z)>0,\textup{Im}(z)\ge 0\}.
 \end{equation}
It is easy to see that $(\Gamma,i\cdot\Gamma,-\Gamma,-i\cdot\Gamma)$ is
a partition of $\mathbb{Z}[i]\setminus\{0\}$, where $i\cdot\Gamma=\{iz\colon\,z\in \Gamma\}$. Thus, for any nonzero Gaussian integer $z$, exactly one of its four associates lies in $\Gamma$.
\begin{definition}\textup{
Let $U$ be a  Gaussian rational unitary matrix. The  \emph{level} of  $U$ is the Gaussian integer $\ell=\ell(U)\in \Gamma$ such that $\ell U$ is a Gaussian integral matrix and $N(\ell)$ is minimal.}
\end{definition}
The assumption $\ell(U)\in \Gamma$ makes $\ell(U)$  unique and hence well-defined. We will make similar convention on least common multiple (LCM) and greatest common divisor (GCD) on Gaussian integers.   We note that $\ell$ is the lcm of all denominators (in the form of reduced fraction) of all  entries in $U$. In particular, if $gU$ is a Gaussian integral matrix then  $\ell\mid g$.  Clearly,  a  Gaussian rational unitary matrix $U$ with $Ue=e$ is a permutation matrix if and only if $\ell(U)=1$.
\begin{theorem}\label{realell}
Let $G\in\mathcal{G}_n^{sc}$ such that $\det W(G)\neq 0$. For any $U\in \mathscr{U}_G$, $\ell(U)$ and $\overline{\ell(U)}$ are associates, that is, $\ell(U)=a$ or  $\ell(U)=a(1+i)$ for some positive integer $a$.
\end{theorem}
\begin{proof}
Let $H\in \mathcal{G}_n^{sc}$ such that $U\in  \mathscr{U}_G(H)$. Thus, $U=W(G)W^{-1}(H)$. Since both $G$ and $H$ are self-converse, it follows from Theorem \ref{Wrealorimag} that there exist two permutation matrices $P$ and $Q$ such that $\overline{W(G)}=P^{-1}W(G)$ and $\overline{W(H)}=Q^{-1}W(H)$. Therefore,
$$\overline{U}= \overline{W(G)W^{-1}(H)}=P^{-1}W(G)W(H)^{-1}Q=P^{-1}UQ.$$
Thus,
$\overline{\ell(U)}U=\overline{\ell(U)\overline{U}}=P^{-1}\overline{\ell(U)U}Q$
and hence $\overline{\ell(U)}U$ is Gaussian integral. Moreover, due to the minimality of $\ell(U)$, we have $\ell(U)\mid \overline{\ell(U)}$. Taking conjugate we have $\overline{\ell(U)} \mid \ell(U)$ and hence $\ell(U)$ and $\overline{\ell(U)}$ are associates. Since $\ell(U)\in \Gamma$, we find that the amplitude of $\ell(U)$ is either $0$ or $\frac{\pi}{4}$. That is,  $\ell(U)=a$ or  $\ell(U)=a(1+i)$ for some positive integer $a$.
\end{proof}
\section{Some divisibility relations}
The following lemma was first established for ordinary graphs in~\cite{wang2006PhD}. The original proof can be easily extended to mixed graphs. For simplicity, we shall write $A=A(G)$ and $W=W(G)$ in the rest of the paper.
\begin{lemma}\label{totalsumAk}
Let $G\in \mathcal{G}_n$. Then for any positive integer $k$, $ e^*A^ke\equiv 0\pmod 2.$
\end{lemma}
\begin{proof}
Denote the $(i,j)$-entry of $A^k$ as $a^{(k)}_{i,j}$. Note that $A^k$ is Hermitian as $A$ is Hermitian. Thus, we have
$$e^* A^k e=\Tr(A^k)+\sum_{i<j}\left(a^{(k)}_{i,j}+\overline{a^{(k)}_{i,j}}\right)=\Tr(A^k)+\sum_{i<j}2\Re( a^{(k)}_{i,j})\equiv \Tr(A^k)\pmod 2. $$
On the other hand, as all diagonal entries of  $A$ are zero, we have
$$\Tr(AA^{k-1})=\sum_{i\neq j}a^{(1)}_{i,j}a^{(k-1)}_{j,i}=\sum_{i<j}\left(a^{(1)}_{i,j}a^{(k-1)}_{j,i}+\overline{a^{(1)}_{i,j}a^{(k-1)}_{j,i}}\right)=\sum_{i<j}2\Re(a^{(1)}_{i,j}a^{(k-1)}_{j,i})\equiv 0\pmod 2.$$
Therefore, $e^*A^k(G)e\equiv 0\pmod 2.$ This proves the lemma.
\end{proof}
Let $p$ be a Gaussian prime and $M$ be a Gaussian integral matrix, we use $\rank_{p}(M)$ to denote the rank of $M$ over the field $\mathbb{Z}[i]/(p)$. Note that $\rank_{\overline{p}}(\overline{M})=\rank_{p}(M)$ always holds. In addition, if $p$ and $\overline{p}$ are associate  then we have $\rank_{\overline{p}}(M)=\rank_{p}(M)$. In particular, $\rank_{1+i}(M)=\rank_{1-i}(M)$.
\begin{corollary}\label{detrank}
Let $G\in \mathcal{G}_n$. Then the followings hold.\\
(1) $2^{\lfloor\frac{n}{2}\rfloor}\mid \det W$, and\\
(2) $\rank_{1+i} W\le \lceil\frac{n}{2}\rceil$.
\end{corollary}
\begin{proof}
Let $M=W^*W$. Let $m_{i,j}$ denote the $(i,j)$-entry of $M$. Then $m_{i,j}=e^*A^{i+j-2}e$. Note that $m_{1,1}=n$. Thus, by Lemma \ref{totalsumAk}, $m_{i,j}\equiv 0\pmod  2$ unless  $(i,j)=(1,1)$ and $n$ is odd. Therefore, $2^n\mid \det M$ when $n$ is even, and $2^{n-1}\mid \det M$ when $n$ is odd. In other words,
\begin{equation}\label{2detM}
2^{2\lfloor\frac{n}{2}\rfloor}\mid \det M.
\end{equation}
 As $2$ and $(1+i)^2$ are associates and $\det M=\det W^*\det W=\overline{\det W}\det W$, we can rewrite (\ref{2detM}) as
\begin{equation}\label{2detconWW}
(1+i)^{4\lfloor\frac{n}{2}\rfloor}\mid \overline{\det W}\det W.
\end{equation}
 As $\overline{1+i}$ and $1+i$ are associates, from (\ref{2detconWW}), we have
 \begin{equation}\label{2detconWW1}
(1+i)^{4\lfloor\frac{n}{2}\rfloor}\mid (\det W)^2
\end{equation}
and hence $(1+i)^{2\lfloor\frac{n}{2}\rfloor}\mid \det W$, i.e., $2^{\lfloor\frac{n}{2}\rfloor}\mid \det W$. This proves (1).

Note that $m_{i,j}\equiv 0\pmod{1+i}$ unless $(i,j)=(1,1)$ and $n$ is odd. We have
\begin{equation}\label{rankM}
\rank_{1+i} M= \begin{cases}
   0 &\mbox{if $n$ is even,}\\
   1  &\mbox{if $n$ is odd.}
   \end{cases}
\end{equation}
Using the familiar inequality that $\rank B+\rank C\le n+\rank BC$ for any matrices of order $n$, we have
\begin{equation}\label{rW2}
\rank_{1+i}W^*+\rank_{1+i}W\le n+\rank_{1+i} M.
\end{equation}
Note that $\rank_{1+i}W^*=\rank_{1+i}W$, which combining with (\ref{rW2}) implies
\begin{equation}\label{rWform}
 \rank_{1+i}W\le \Big\lfloor\frac{n+\rank_{1+i} M}{2}\Big\rfloor.
 \end{equation}
Clearly, using (\ref{rankM}), the right term in (\ref{rWform}) can be reduced to $\lceil\frac{n}{2}\rceil$. This proves (2).
\end{proof}

\begin{lemma}\label{firstr}
Let $G\in\mathcal{G}_n$ and $r=\rank_{1+i}W$. Then the first $r$ columns of $W$ are linearly independent over $\mathbb{Z}[i]/(1+i)$.
\end{lemma}
\begin{proof}
Suppose to the contrary that $e,Ae,\ldots,A^{r-1}e$ are linearly dependent. Then there exists an integer $m$ such that $m\le r-1$ and
\begin{equation}\label{span}
A^m e\in \Span\{e,Ae,\ldots,A^{m-1}e\}.
\end{equation}
Using (\ref{span}) twice, we have
\begin{equation}\label{span2}
A^{m+1} e\in \Span\{Ae,A^2e,\ldots,A^{m}e\}\subseteq \Span\{e,Ae,\ldots,A^{m-1}e\}.
\end{equation}
Similarly, for any $m'>m$, we always have $A^{m'} e\in \Span\{e,Ae,\ldots,A^{m-1}e\}.$ Thus,
$\rank_{1+i} W\le m<r=\rank_{1+i} W$. This contradiction completes the proof of this lemma.
\end{proof}
The following result gives a basic relation between $\rank_{1+i}W$ and $\det W$. We note that the real counterpart is easy to obtain using Smith Normal Form and the fact that $2$ is a prime in $\mathbb{Z}$. Unfortunately, similar argument is not valid since $2$ is factorable in $\mathbb{Z}[i]$. Some new techniques have to be used to overcome this difficulty.
\begin{lemma}\label{detvsrank}
Let $G\in \mathcal{G}_n^{sc}$ and $r=\rank_{1+i}W$. Then we have
\begin{equation}\label{rankmid}
2^{n-r}\mid \det W.
\end{equation}
\end{lemma}
\begin{proof}
By Corollary \ref{detrank}, we have $r\le \lceil\frac{n}{2}\rceil$ and $2^{\lfloor{\frac{n}{2}}\rfloor}\mid\det W$. If $r=\lceil\frac{n}{2}\rceil$ then $n-r=\lfloor{\frac{n}{2}}\rfloor$ and hence (\ref{rankmid}) holds. Thus, it suffices to consider the case that $r<\lceil\frac{n}{2}\rceil$.

By Lemma \ref{firstr}, $A^re$ can be expressed as a linear combination of $A^{r-1}e,A^{r-2}e,\ldots,e$ over $\mathbb{Z}[i]/(1+i)$. Let $(c_1,c_2,\ldots,c_r)$ be a 0-1 vector such that  $A^re\equiv c_1A^{r-1}e+c_2A^{r-2}e+\cdots+c_re\pmod{1+i}$. Let $B=A^r-c_1A^{r-1}-c_2A^{r-2}-\cdots-c_rI$. Note that $B$ is Hermitian.

\noindent\textbf{Claim 1}: $e^*BA^kBe\equiv 0\pmod{4}$ for any $k\ge 0$.

As $Be\equiv 0\pmod{1+i}$, $\frac{Be}{1+i}$ is a Gaussian integral vector. Write $g=\frac{Be}{1+i}$. Now we have $e^*BA^kBe=2g^*A^kg$. Thus, it suffice to show that $g^*A^kg\equiv 0\pmod{2}$. We consider two cases:\\
\noindent\emph{Case} 1: $k$ is odd, say, $k=2s+1$.

Write $A^sg=h=(h_1,h_2,\ldots,h_n)$. Note that $A$ is Hermitian with vanishing diagonal entries. We have
 $$g^*A^kg=h^*Ah=\sum_{s\neq t}\overline{h_s}a_{s,t}h_t=\sum_{s< t}(\overline{h_s}a_{s,t}h_t+\overline{h_t}a_{t,s}h_s)=\sum_{s< t}(\overline{h_s}a_{s,t}h_t+\overline{\overline{h_s}a_{s,t}h_t})\equiv 0\pmod 2.$$

 \noindent\emph{Case} 2: $k$ is even, say, $k=2s$.

As $g^*A^kg$ is real, we only need to prove $g^*A^kg\equiv 0\pmod{1+i}$. Note that for any Gaussian integer $c$, $\overline{c}c\equiv c\pmod{1+i}$. Thus,
 \begin{equation}
 g^*A^kg=(A^sg)^*(A^sg)\equiv e^*(A^sg)\pmod{1+i}.
 \end{equation}
  As  $g=\frac{Be}{1+i}$, we are done if
    \begin{equation}\label{mod2for}
  e^*A^sBe\equiv 0\pmod 2.
  \end{equation}
By Lemma \ref{totalsumAk}, $e^*A^je\equiv 0\pmod{2}$ for any $j\ge 1$. Thus, if $s\ge 1$ then  (\ref{mod2for}) holds as $B$ is a linear combination of $A^0,A^1,\ldots,A^{r-1}$. If $s=0$ then  $e^*A^sBe= e^*Be\equiv 0\pmod{1+i}$ as $Be\equiv 0\pmod{1+i}$. Note that $e^*Be$ is real. This implies that $e^*Be\equiv 0\pmod{2}$ and hence (\ref{mod2for}) always holds.

  Combining Case 1 and Case 2, Claim 1 follows.

  Define $\hat{W}=(e,Ae,\ldots,A^{r-1}e,Be,BAe,\ldots,BA^{n-r-1})$. Clearly,
  \begin{equation}\label{equdet1}
  \det \hat{W}=\det W.
  \end{equation}
   Write
$\hat{W}=(W_1,W_2)$, where,
\begin{equation}
W_1=(e,Ae,\ldots,A^{r-1}e)\quad\textup{and}\quad W_2=B(e,Ae,\ldots,A^{n-r-1}e).
\end{equation}  Now we have
\begin{equation}
\hat{W}^*\hat{W}=(\hat{W}^*W_1,\hat{W}^*W_2)=\left(\begin{matrix}
W_1^*W_1&W_1^*W_2\\
W_2^*W_1&W_2^*W_2\\
\end{matrix}
\right).
\end{equation}
Write $W_1^*W_2=(f^{(1)},f^{(2)},\ldots,f^{(n-r)})$. Note that $n-r\ge r+1$ as $r<\lceil\frac{n}{2}\rceil$.

\noindent\textbf{Claim} 2: $f^{(j)}-c_1f^{j-1}-c_2f^{(j-2)}-\cdots-c_rf^{(j-r)}\equiv 0\pmod{4}$ for $j=r+1,r+2,\ldots,n-r$.

Let $f^{(s)}_t$ denote the $t$-entry of $f^{(s)}$ for $s\in\{1,2,\ldots,n-r\}$ and $t\in\{1,2,\ldots,r\}$. Note that $A$ and $B$ are commutative. One easily finds that $f^{(s)}_t=e^*BA^{t+s-2}e$. Thus,
\begin{eqnarray*}
&&f^{(j)}_t-c_1f^{j-1}_t-c_2f^{(j-2)}_t-\cdots-c_rf^{(j-r)}_t\\
&=&e^*B(A^{t+j-2}-c_1A^{t+j-3}-\cdots-c_rA^{t+j-r-2})e\\
&=&e^*BA^{t+j-r-2}(A^{r}-c_1A^{r-1}-\cdots-c_rA^{0})e\\
&= &e^*BA^{t+j-r-2}Be\label{z0WHz0}\\
&\equiv& 0\pmod{4},
\end{eqnarray*}
where the last congruence follow from Claim 1. This proves Claim 2.

Let $F^{(j)}$ denote the $j$-th column of $\hat{W}^*W_2$ for $j=r+1,r+2,\ldots,n-r$, that is,
$$F^{(j)}=\left(\begin{matrix}
f^{(j)}\\
q^{(j)}\\
\end{matrix}\right),$$
where $q^{(j)}$ denotes the $j$-th column of $W_2^*W_2$. Define $\hat{F}^{(j)}=F^{(j)}-c_1F^{(j-1)}-c_2F^{(j-2)}-\cdots-c_rF^{(j-r)}$.  Note that by Claim 1, each entry of $W_2^*W_2$ is a multiple of $4$. Combining this fact with Claim 2 we find that $\hat{F}^{(j)}\equiv 0\pmod 4$. Let
$$M=(\hat{W}^*W_1,{F}^{(1)},{F}^{(2)},\ldots,{F}^{(r)},\hat{F}^{(r+1)},\hat{F}^{(r+2)},\ldots,\hat{F}^{(n-r)}).$$
As $\hat{W}^*\hat{W}=(\hat{W}^*W_1,\hat{W}^*W_2)=(\hat{W}^*W_1,{F}^{(1)},{F}^{(2)},\ldots,{F}^{(n-r)})$, one easily finds that $\det \hat{W}^*\hat{W}=\det M$.  Together with (\ref{equdet1}), we have
\begin{equation}
\det M=\overline{\det W}\det W.
\end{equation}

By Lemma \ref{totalsumAk}, all entries of the real matrix $(\hat{W}^*W_1,{F}^{(1)},{F}^{(2)},\ldots,{F}^{(r)})$ is even,  except the upper left corner when $n$ is odd. As $\hat{F}^{(j)}\equiv 0\pmod 4$ for each $j\in \{r+1,r+2,\ldots,n-r\}$, we find that
$2^{2r}4^{n-2r}\mid \det M$ when $n$ is even and $2^{2r-1}4^{n-2r}\mid \det M$ when $n$ is odd. If $n$ is even then $(1+i)^{4n-4r}\mid \overline{\det W}\det W$ and hence $(1+i)^{2n-2r}\mid \det W$, i.e., $2^{n-r}\mid \det W$. If $n$ is odd then
$(1+i)^{4n-4r-2}\mid \overline{\det W}\det W$ and hence $(1+i)^{2n-2r-1}\mid \det W$, i.e., $2^{n-r-1}(1+i)\mid \det W$. Fortunately, since $G$ is self-converse, $\det W$ is real or pure imaginary  by Lemma \ref{Wrealorimag}. Now $2^{n-r-1}(1+i)\mid \det W$ is equivalent to $2^{n-r}\mid \det W.$ This completes the proof of this lemma.
\end{proof}
\begin{corollary}\label{rankceil}
Let $G\in\mathcal{G}_n^{sc}$. If $\frac{\det W}{2^{\lfloor\frac{n}{2}\rfloor}}$ is odd, then $\rank_{1+i}W=\lceil\frac{n}{2}\rceil$.
\end{corollary}
\begin{proof}
Let $r= \rank_{1+i}W$. By Lemma \ref{detrank}, $r\le \lceil\frac{n}{2}\rceil$. Suppose to the contrary that $r< \lceil\frac{n}{2}\rceil$. Then $n-r\ge \lfloor\frac{n}{2} \rfloor+1$. By Lemma \ref{detvsrank}, we have $2^{ n-r}\mid \det W$ and hence $2^{\lfloor\frac{n}{2} \rfloor+1}\mid \det W$. This is a contradiction.
\end{proof}
An $n\times n$ Gaussian integral  matrix $U$  is called \emph{unimodular} if $\det U$ is a unit in $\mathbb{Z}[i]$, i.e., $\det U=i^k$ for some $k\in[4]$. The following result is well known.
\begin{lemma}\label{smithdec}
For every $n\times n$ Gaussian integral matrix $M$ with full rank, there exist unimodular matrices $V_1$ and $V_2$ such that $M=V_1SV_2$, where $S=\diag(d_1,d_2,\ldots,d_n)$ is a Gaussian integral diagonal matrix with $d_i$ being the $i$-th entry in the diagonal and $d_i\mid d_{i+1}$ for $i=1,2,\ldots,n-1$.
\end{lemma}
For a Gaussian integral matrix $M$, the above $S$ is called the \emph{Smith Normal Form} (SNF for short) of $M$ and $d_i$ is called the $i$-th \emph{elementary divisor} of the matrix. The $i$-th elementary divisor is unique up to multiplication by units. For clarity, we always assume $d_i\in \Gamma\cup\{0\}.$
The following lemma appeared in \cite{wang2013ElJC} for ordinary integral matrix. The proof given in  \cite{wang2013ElJC} is of course valid for Gaussian integral matrix. We include the short proof here for the convenience of readers.
\begin{lemma}\cite{wang2013ElJC}\label{psolutionanddn}
Let $p$ be a Gaussian prime and $M$ be an $n\times n$ Gaussian integral matrix. Then  $Mz\equiv 0\pmod{p^2}$ has a solution $z\not\equiv 0\pmod{p}$ if and only if $p^2\mid d_n.$
\end{lemma}
\begin{proof}
Let $V_1$ and $V_2$ be unimodular matrices such that  $M=V_1\diag(d_1,d_2,\ldots,d_n)V_2$. The equation $Mz\equiv 0\pmod{p^2}$ is equivalent to $\diag(d_1,d_2,\ldots,d_n)V_2z\equiv 0\pmod{p^2}$. Let $y=V_2z$. Consider $\diag(d_1,d_2,\ldots,d_n)y\equiv 0\pmod{p^2}$. If $p^2\mid d_n$, let $y=(0,0,\ldots,1)^\T$, then $z=V_2^{-1}y\not\equiv 0\pmod{p}$ is a required solution to the original congruence equation. On the other hand, if $p^2\not\mid d_n$, then the fact that $p$ is Gaussian prime implies that $\diag(d_1,d_2,\ldots,d_n)y\equiv 0\pmod{p^2}$ has no solution $y$ with $y\not\equiv 0\pmod{p}$, i.e., $\diag(d_1,d_2,\ldots,d_n)V_2z\equiv 0\pmod{p^2}$ has no solution of $z$ with $z\not\equiv 0\pmod{p}$.
\end{proof}
\begin{lemma}\label{sameSNF}
Let $G\in \mathcal{G}_n^{sc}$. Then for each $i\in\{1,2,\ldots,n\}$, $d_i$ and $\overline{d_i}$ are associates, where $d_i$ is the $i$-th elementary divisor of $W$. In particular, $W$ and $W^*$ have the same \textup{SNF}.
\end{lemma}
\begin{proof}
 By Lemma \ref{smithdec}, there exist unimodular matrices $S$ and $T$ such that $W=S\Lambda T$, where $\Lambda=\diag(d_1,d_2,\ldots,d_n)$ is the SNF of $W$. Thus, $\overline{W}=\overline{S}~\overline{\Lambda}~\overline{T}$. On the other hand,
 by Theorem \ref{Wrealorimag}, there exists a permutation matrix $P$ such that $\overline{W}=P^{-1}W$. Thus,  $\overline{W}=(P^{-1}S)\Lambda T$. Therefore, $\overline{d_i}$ and $d_i$ are associates. Finally, as $W^*=(\overline{W})^\T=T^\T\Lambda (S^\T P)$,  $W$ and $W^*$ have the same SNF.
\end{proof}
\section{Main result}
For convenience, we restate Conjecture \ref{mainconj} in an equivalent form. Note that for $z=au$, where $a$ is a nonnegative integer in $\mathbb{Z}$ and $u$ is a unit in $\mathbb{Z}[i]$,  $z$ is square-free in $\mathbb{Z}[i]$ if and only if $a$ is odd and square-free in $\mathbb{Z}$.
\begin{conjecture}\label{mainconj2}
Let $G\in \mathcal{G}_n^{sc}$. If $\frac{\det W}{2^{\lfloor n/2 \rfloor}}$ is square-free in $\mathbb{Z}[i]$,  then for any $U\in \mathscr{U}_G$, $\ell(U)=1$.
 \end{conjecture}

The main result of this paper is the following
\begin{theorem}\label{main}
Let $G\in \mathcal{G}_n^{sc}$. If $\frac{\det W}{2^{\lfloor n/2 \rfloor}}$ is  square-free in $\mathbb{Z}[i]$, then for any $U\in \mathscr{U}_G$, $\ell(U)\in \{1,1+i\}$.
\end{theorem}
\subsection{The case $p$ is odd}
\begin{theorem}\label{mainodd}
Let $G\in \mathcal{G}_n^{sc}$ such that $\det W\neq 0$. Let $U\in \mathscr{U}_G$ with level $\ell$. For any odd Gaussian prime $p$, if $p^2\not\mid \det W$ then $p\not\mid \ell$.
\end{theorem}

\begin{lemma}\label{ellz0}
Let $G\in \mathcal{G}_n^{sc}$ such that $\det W\neq 0$.  Let $U\in \mathscr{U}_G$ with level $\ell$. Then we have\\
(\romannumeral1) $\ell\mid d_n$, where $d_n$ is the $n$-th elementary divisor of the \textup{SNF} of $W$.\\
(\romannumeral2) Let $p\in\Gamma$ be any odd prime factor of $d_n$. If $p\mid \ell$ and $\rank_p W=n-1$ then there exists a Gaussian integral vector $z_0\not\equiv 0\pmod{p}$, $z_0\not\equiv 0\pmod{\overline{p}}$ and a Gaussian integer $\lambda_0$ such that
\begin{equation}
z_0^*A^kz_0\equiv 0\pmod{N(\lcm(p,\overline{p}))},~\text{ for any}~ k\ge 0,
\end{equation}
\begin{equation}
 W^*z_0\equiv 0\pmod{\lcm(p,\overline{p})},
 \end{equation}
 and
 \begin{equation}\label{Az0}
Az_0\equiv \lambda_0 z_0\pmod{\lcm (p,\overline{p})}.
\end{equation}
In particular, $z_0^*z_0\equiv 0\pmod{N(\lcm(p,\overline{p}))}$ and $e^*z_0\equiv 0\pmod{\lcm(p,\overline{p})}.$
\end{lemma}
\begin{proof}
Let $H\in \mathcal{G}_n^{sc}$ such that $\mathscr{U}_G=\{U\}$. By Theorem \ref{reduce2unitary}, we have $U=W(G)W^{-1}(H)$. As $U^*U=I$,  $U=(U^*)^{-1}=(W^*(G))^{-1}W^*(H)$. By Lemma \ref{smithdec}, there exist unimodular matrices $S$ and $T$ such that $W^*(G)=S\Lambda T$, where $\Lambda=\diag(d_1,d_2,\ldots,d_n)$ is the SNF of $W^*(G)$ (or $W(G)$ equivalently due to Lemma \ref{sameSNF}). Now we can write
$$U=T^{-1}\diag({d_1}^{-1},d_2^{-1},\ldots,d_n^{-1})S^{-1}W^*(H).$$
As $T^{-1}$, $S^{-1}$, $d_n\diag({d_1}^{-1},d_2^{-1},\ldots,d_n^{-1})$ and $W^*(H)$ are Gaussian integral, we see that $d_nU$ is Gaussian integral, and hence $\ell\mid d_n$. This proves (\romannumeral1).

Let $U_1=\ell U$. To show (\romannumeral2), we consider the following two cases:

\noindent\emph{Case 1}: $p$ and $\overline{p}$ are associates.

In this case, $\lcm(p,\overline{p})=p$. By the definition of $\ell$, $U_1$ is Gaussian integral and $U_1$ contains a column $z_0$ such that $z_0\not\equiv 0\pmod{p}$. Since $U^*A^k(G)U=A^k(H)$ for any $k\ge 0$, we have
 $U_1^* A^k(G) U_1=\overline{\ell}\ell A^k(H)\equiv 0\pmod {N(\ell)}$, which implies  $U_1^* A^k(G) U_1\equiv 0\pmod {N(p)}$ since $N(p)\mid N(\ell)$. Therefore, $z_0^*A^k(G)z_0\equiv 0\pmod{N(p)}$.  As  $W^*(G) U_1=\ell W^*(H)$, we have $W^*(G)z_0\equiv 0\pmod \ell$ and hence $W^*(G)z_0\equiv 0\pmod p$.

 By Lemma~\ref{sameSNF}, we have $\rank_{p}W^*(G)=\rank_{p}W(G)$ and hence $\rank_{p}W^*(G)=n-1$. As $W^*(G)U_1\equiv 0\pmod{p}$ and $U_1$ contains a column $z_0\not\equiv 0\pmod{p}$, we see that $\rank_{p}U_1=1$ and hence there exists a Gaussian integral row vector $\gamma$ such that $U_1\equiv z_0\gamma\pmod {p}$. Suppose that $z_0$ is the $t$-th column of $U_1$. As $A(G)U_1=U_1A(H)$, we have
 \begin{equation}\label{pvector}
 A(G)z_0\equiv z_0(\gamma A_t(H))\equiv \lambda_0 z_0\pmod{p},
\end{equation}
 where $ A_t(H)$ is the $t$-th column of $A(H)$ and $\lambda_0=\gamma A_t(H)$ is a Gaussian integer.

\noindent\emph{Case 2}: $p$ and $\overline{p}$ are not associates.

In this case, $\lcm(p,\overline{p})=p\overline{p}=N(p)$. By Theorem \ref{realell}, $\ell$ and $\overline{\ell}$ are associates. Since $p\mid \ell$ we have $\overline{p}\mid \overline{\ell}$ and hence $\overline{p}\mid \ell$. Therefore, $N(p)\mid \ell$. Note that $W^*(G)U_1=\ell W^*(H)$ and $W^*(G)=S\Lambda T$. We have $S\Lambda T U_1\equiv 0 \pmod {N(p)}$, which can be simplified to
\begin{equation}\label{TU}
\Lambda T U_1\equiv 0 \pmod {N(p)}
\end{equation}
since $S$ is unimodular.

Write $\tilde{U}=TU_1$ and $e_n=(0,\ldots,0,1)^\T$, an $n$-dimensional coordinate vector. As $\rank_p W(G)=n-1$, we must have that  $p\not\mid d_i$ for any $i\in\{1,2,\ldots,n-1\}$. Also  $\overline{p}\not\mid d_i$ for any $i\in\{1,2,\ldots,n-1\}$ as $d_i$ and $\overline{d_i}$ are associates. It follows from (\ref{TU}) that
\begin{equation}\label{UNp}
\tilde{U}\equiv(m_1e_n,m_2e_n,\ldots,m_n e_n)\pmod {N(p)}
\end{equation}
for some Gaussian integers $m_1,m_2,\ldots,m_n$. Write $u=T^{-1}e_n$. Then we have
\begin{equation}\label{U1Np}
U_1\equiv(m_1u,m_2u,\ldots,m_nu)\pmod {N(p)}.
 \end{equation}
If $p\mid m_i$ for each $i\in \{1,2,\ldots,n\}$ then $\frac{U_1}{p}$ is Gaussian integral, contradicting the minimality of $\ell$. Thus, $p\not\mid m_i$ for some $i\in \{1,2,\ldots,n\}$. Similarly, $\overline{p}\not\mid m_j$ for some $j\in \{1,2,\ldots,n\}$. Also $u\not\equiv 0\pmod p$ and $u\not\equiv 0\pmod {\overline{p}}$. Denote $c=\gcd(m_1,m_2,\ldots,m_n)$. Then $p\not\mid c$ and $\overline{p}\not\mid c$. Since the ring of Gaussian integers is Euclidian, there exist $n$ Gaussian integers $q_1,q_2,\ldots,q_n$ such that $c=q_1m_1+q_2m_2+\cdots+q_nm_n$. Write $q=(q_1,q_2,\ldots,q_n)^\T$ and let $z_0=U_1q$.

From (\ref{U1Np}), we have
\begin{equation}\label{z0Np}
z_0\equiv (m_1 u,m_2 u,\ldots,m_n u)q\equiv cu\pmod {N(p)}.
 \end{equation}
 Therefore, $z_0\not\equiv 0\pmod{p}$ and $z_0\not\equiv 0\pmod{\overline{p}}$.  Note that $z_0^*A^k(G)z_0=(U_1q)^*A^k(G)(U_1q)=q^*U_1^*A^k(G)U_1q=l\overline{l}q^*U^*A^k(G)Uq=l\overline{l}q^*A^k(H)q$. We have
$z_0^*A^k(G)z_0\equiv 0\pmod{N^2(p)}$ as $N(p)\mid \ell$. As $W^*(G)z_0=W^*(G)U_1q=\ell W^*(G)Uq=\ell W^*(H)q$,  we have $W^*(G)z_0\equiv 0\pmod{N(p)}$.

Let $\eta=(\frac{m_1}{c},\frac{m_1}{c},\ldots,\frac{m_n}{c})$. It follows from (\ref{U1Np}) and (\ref{z0Np}) that $U_1\equiv z_0 \eta\pmod{N(p)}$. Now the same argument as in Case 1 shows that $A(G)z_0\equiv \lambda_0 z_0\pmod{N(p)}$ for some Gaussian integer $\lambda_0$. This finally completes the proof.
\end{proof}

 \begin{lemma}\label{real}
 Using the notations of Lemma \ref{ellz0}, $\Im(\lambda_0)\equiv 0\pmod{\lcm(p,\overline{p})}$.
 \end{lemma}
 \begin{proof}
 We first show $\Im(\lambda_0)\equiv 0\pmod{p}$. From  (\ref{Az0}), we have
 \begin{equation}\label{Az0pp}
 Az_0\equiv \lambda_0z_0\pmod{p}\quad \text{and}~Az_0\equiv \lambda_0z_0\pmod{\overline{p}}.
 \end{equation}
Taking conjugation on both side of the second congruence equation in (\ref{Az0pp}), we have
 \begin{equation}\label{conpvector}
 \overline{A}\overline{z_0}\equiv \overline{\lambda_0} \overline{z_0}\pmod{p}.
\end{equation}
Since $G$ is self-converse, there exists a permutation matrix $P$ such that $\overline{A}=P^{-1}AP$. Thus, by (\ref{conpvector}),
 $P^{-1}AP\overline{z_0}\equiv \overline{\lambda_0} \overline{z_0}\pmod{p}$ and hence $AP\overline{z_0}\equiv \overline{\lambda_0} P\overline{z_0}\pmod{p}$. Write $z_1=P\overline{z_0}$. Suppose to the contrary that $\Im (\lambda_0)\not\equiv 0\pmod{p}$. Then $\lambda_0\not\equiv \overline{\lambda_0}\pmod{p}$ as $p$ is odd. As $z_0$ and $z_1$ are eigenvectors corresponding to different eigenvalues, they are linearly independent over $\mathbb{Z}[i]/(p)$. Moreover, $e^*z_1=e^*P\overline{z_0}=e^*\overline{z_0}=\overline{e^*z_0}\equiv 0\pmod{p}$ and hence $e^*A^kz_1\equiv\left(\overline{\lambda_0}\right)^k e^*z_1\equiv 0\pmod{p}$. Therefore, $W^*z_1\equiv 0\pmod{p}$. As  $W^*z_0\equiv 0\pmod{p}$ and $z_0,z_1$ are linearly independent, we have  $\rank_{p}W^*\le n-2$, i.e., $\rank_{p}W\le n-2$ by Lemma \ref{sameSNF}.  This contradicts our assumption that $p^2\not\mid\det W$.

 Note that the above argument also holds if we interchange $p$ and $\overline{p}$. Thus we also have  $\Im(\lambda_0)\equiv 0\pmod{\overline{p}}$ and hence $\Im(\lambda_0)\equiv 0\pmod{\lcm(p,\overline{p})}$.
 \end{proof}
 \begin{lemma}\label{yz}
 If $(A-\lambda_0 I)y\equiv sp^jz_0\pmod {p^{j+1}}$ for some $s\in \mathbb{Z}[i]$ and $j\in \{0\}\cup \mathbb{Z}^+$ then
 \begin{equation}W^*y\equiv e^*y(1,\lambda_0,\ldots,\lambda_0^{n-1})^\T\pmod{p^{j+1}}.
 \end{equation}
 \end{lemma}
 \begin{proof} We claim that $e^*A^ky\equiv\lambda_0^ke^*y\pmod{p^{j+1}}$ for  $k= 0,1,\ldots,n-1$. The case $k=0$ is trivial. Let $k<n-1$. Suppose that the claim holds for $k$ and we are going to check it for $k+1$. By Lemma \ref{ellz0}, $W^*z_0\equiv 0\pmod p$, and hence $e^*A^kz_0\equiv 0\pmod{p}$ as $e^*A^{k}$ is the $(k+1)$-th row of $W^*$. Thus,  $e^*A^{k}sp^jz_0\equiv 0\pmod{p^{j+1}}$. Now, by the condition of this lemma and induction hypothesis,
 $$e^*A^{k+1}y\equiv e^*A^{k}(sp^jz_0+\lambda_0y)\equiv \lambda_0e^*A^{k}y\equiv \lambda_0^{k+1}e^*y\pmod{p^{j+1}}.$$
 This proves the claim and the lemma follows.
 \end{proof}
 \begin{lemma}\label{allsolution}
  Using the notations of Lemma \ref{ellz0},
   \begin{equation}\label{rankHandz0}
  z_0^*(A-\lambda_0I,z_0)\equiv 0\pmod{p} \quad\textup{and}\quad\rank_p  (A-\lambda_0I,z_0)= n-1.
 \end{equation}
 \end{lemma}
\begin{proof}
 We claim that
\begin{equation}\label{rankn2}
\rank_p (A-\lambda_0 I)\ge n-2.
 \end{equation}
Suppose to the contrary that $\rank_p(A-\lambda_0 I)\le n-3.$ Consider the equation $(A-\lambda_0 I)z\equiv 0\pmod {p}$ which has a nontrivial solution $z_0$.  There are at least two nontrivial solution $y_1$ and $y_2$ such that $z_0,y_1,y_2$ are linear independent over $\mathbb{Z}[i]/(p).$ If either of $y_1$ and $y_2$, say $y_1$, satisfies  $e^*y_1\equiv 0\pmod {p}$, then it follows from Lemma \ref{yz} for $s=0$ that $z_0$ and $y_1$ are two linear independent solutions of $W^*z\equiv 0\pmod{p}$, which contradicts the fact that $\rank_p W^*=n-1$. Thus,  $e^*y_1\not\equiv 0\pmod {p}$ and $e^*y_2\not\equiv 0\pmod{p}$. Let $y_3=(e^*y_1)y_2-(e^*y_2)y_1$. As $y_1$ and $y_2$ are linear independent, $y_3\not\equiv 0\pmod{p}$. Note that $e^*y_3=0$. Thus,  $W^*y_3\equiv 0\pmod{p}$. Therefore, $W^*z\equiv 0\pmod{p}$ has  solutions $z$ and $y_3$, which are clearly independent over $\mathbb{Z}[i]/(p).$ This contradiction completes the proof of (\ref{rankn2}).

Next we show  (\ref{rankHandz0}). By Lemma \ref{ellz0}, $(A-\lambda_0 I)z_0\equiv 0\pmod {\overline{p}}$. Taking conjugate transpose and noting that $A^*=A$ and $\Im(\lambda_0)\equiv 0\pmod {p}$, we have
  $z_0^*(A-\lambda_0 I)\equiv 0\pmod {p}$. Combining with the fact that $z_0^*z_0\equiv 0 \pmod {p}$, we obtain $z_0^*(A-\lambda_0I,z_0)\equiv 0\pmod{p} $. As $z_0^*\not\equiv 0\pmod{p}$, we have
 \begin{equation}\label{lessthan}
  \rank_p (A-\lambda_0I,z_0) \le n-1.
  \end{equation}

 Suppose to the contrary that the equality in (\ref{lessthan}) does not hold. Then, by (\ref{rankn2}),
 $\rank_p  (A-\lambda_0I)=n-2$ and $z_0$ can be written as the linear combination of the columns of $A-\lambda_0I$, say
 $z_0\equiv (A-\lambda_0I)z_1\pmod{p}$.

 As $\rank_p  (A-\lambda_0I)=n-2$, $(A-\lambda_0I)y\equiv 0\pmod{p}$ has two  solutions $z_2$ and $z_3$ which are independent over $\mathbb{Z}[i]/(p)$. Since $(A-\lambda_0I)z_1\equiv z_0\not \equiv 0\pmod{p}$, $z_1$ can not be written as a linear combination of $z_2$ and $z_3$. This implies that $z_1,z_2,z_3$ are linearly independent. Consider the equation $e^*(k_1z_1+k_2z_2+k_3z_3)\equiv 0 \pmod{p}$ with three unknowns $k_1,k_2,k_3$. Clearly, it has at least two independent solutions over $\mathbb{Z}[i]/(p)$. Let $(a_1,a_2,a_3)^\T$ and  $(b_1,b_2,b_3)^\T$  be such two solutions and write $\alpha=a_1z_1+a_2z_2+a_3z_3$ and $\beta=b_1z_1+b_2z_2+b_3z_3$. It is easy to see that $\alpha$ and $\beta$ are linearly independent over $\mathbb{Z}[i]/(p)$. Note that $(A-\lambda_0I)\alpha\equiv a_1z_0$ and $e^*\alpha\equiv 0\pmod{p}$. It follows from Lemma \ref{yz} that $W^*\alpha\equiv 0\pmod{p}$. Similarly,  $W^*\beta\equiv 0\pmod{p}$. Thus, we have found two independent solutions of $W^*z\equiv 0\pmod{p}$. This contradicts the fact that $\rank_p W^*=n-1$ and hence completes the proof of this lemma.
 \end{proof}
 \begin{lemma}\label{solutionmodp2}
 Using the notations of Lemma \ref{ellz0}, $W^*z\equiv 0\pmod{p^2}$ has a solution $z\not\equiv 0\pmod{p}.$
 \end{lemma}
 \begin{proof}
Note that $p^2\mid N(\lcm(p,\overline{p}))$ always hold. By Lemma \ref{ellz0} , we have
 $z_0^*Az_0\equiv 0\pmod{p^2}$ and $z_0^*z_0\equiv 0\pmod{p^2}$. Consequently,
 \begin{equation}\label{conp2}
 z_0^*(A-\lambda_0 I)z_0\equiv 0\pmod{p^2}.
 \end{equation}
 Note that $(A-\lambda_0 I)z_0\equiv 0\pmod{p}$ by Lemma \ref{ellz0}. We can rewrite (\ref{conp2}) as
 \begin{equation}\label{conp1}
 z_0^*\frac{(A-\lambda_0 I)z_0}{p}\equiv 0\pmod{p}.
 \end{equation}
  Note that $z_0^*\not\equiv 0\pmod{p}$. It follows from Lemma \ref{allsolution} that $\frac{(A-\lambda_0 I)z_0}{p}$ can be expressed as a linear combination of the columns of $(A-\lambda_0 I,z_0)$ over the field $\mathbb{Z}[i]/(p)$. Write
 \begin{equation}
 \frac{(A-\lambda_0 I)z_0}{p}\equiv (A-\lambda_0 I)y+sz_0\pmod{p}.
 \end{equation}
 Multiplying each side by $p$, we have $(A-\lambda_0 I)z_0\equiv (A-\lambda_0 I)py+spz_0\pmod{p^2}$, i.e.,
 \begin{equation}
 (A-\lambda_0 I)(z_0-py)\equiv spz_0\pmod{p^2}.
 \end{equation}
 By Lemma \ref{yz}, we have
 \begin{equation}\label{Wp1}
 W^*(z_0-py)\equiv e^*(z_0-py)(1,\lambda_0,\lambda_0^2,\ldots,\lambda_{0}^{n-1})^\T\pmod{p^2}.
 \end{equation}

\noindent\textbf{Claim}: There exists a vector $z_1$ such that  $e^*z_1\not\equiv 0\pmod {p}$ and $(A-\lambda_0I)z_1\equiv sz_0\pmod p$ for some $s\in\{0,1\}$.

By Lemma \ref{allsolution}, $\rank_p (A-\lambda_0I)=n-2$ or $n-1$. We prove the claim by considering the following two cases:\\
\emph{Case}1: $\rank_p (A-\lambda_0I)=n-2$.

By the condition of this case, the subspace of the eigenvectors of $A$ corresponding to $\lambda_0$ is 2-dimensional. Thus, there exists a vector $z_1$ with $(A-\lambda_0 I)z_1\equiv 0\pmod{p}$ that is linearly independent with $z_0$. If $e^*z_1\equiv 0\pmod{p}$ then $W^*z\equiv 0\pmod{p}$ by Lemma \ref{yz}. Note that $W^*z_0\equiv 0\pmod p$. Thus, $\rank_p W^*\le n-2$, a contradiction. \\
\emph{Case 2}: $\rank_p (A-\lambda_0I)=n-1$.

Since $\rank_p( A-\lambda_0I,z_0)=n-1$, the condition of this case implies that $z_0$ can be expressed as a linear combination of the columns of $A-\lambda_0 I$, i.e., there exists a vector $z_1$ such that $(A-\lambda_0I)z_1\equiv z_0\pmod{p}$. Since $z_0\not\equiv 0\pmod{p}$, $z_1\not\equiv 0\pmod{p}$ and $z_1$  is not a eigenvector of $A$ corresponding to $\lambda$ (over $\mathbb{Z}[i]$).  Thus, $z_1$ and $z_0$ are linearly independent. Finally, we also have  $e^*z_1\not\equiv 0\pmod{p}$ by the same argument as in Case 1. This proves the claim.

Note that $ e^*(z_0-py)\equiv e^*z_0\equiv0\pmod{p}$. By the claim, there exists a Gaussian integer $g$ such that $\frac{e^*(z_0-py)}{p}\equiv ge^*z_1\pmod{p}$, i.e,
\begin{equation}\label{Wp2}
e^*(z_0-py)\equiv gpe^*z_1\pmod{p^2}.
 \end{equation}
 By Lemma \ref{yz}, we have  $W^*z_1\equiv e^*z_1(1,\lambda_0,\ldots,\lambda_0^{n-1})^\T\pmod{p}$ and hence
 \begin{equation}\label{Wp3}
 gpW^*z_1\equiv gpe^*z_1(1,\lambda_0,\ldots,\lambda_0^{n-1})^\T\pmod{p^2}.
 \end{equation}
 It follow from (\ref{Wp1}), (\ref{Wp2}) and (\ref{Wp3}) that $ W^*(z_0-py)\equiv  gpW^*z_1\pmod{p^2},$ i.e., $ W^*(z_0-py-gpz_1)\equiv  0\pmod{p^2}$.
This completes the proof of this lemma as $z_0-py-gpz_1\equiv z_0\not\equiv 0\pmod{p}$.
\end{proof}
\noindent\emph{Proof of Theorem \ref{mainodd}.} We may assume $p\in \Gamma$. Suppose to the contrary that  $p\mid \ell$. Then by Lemma \ref{ellz0}, $\ell\mid d_n$ and hence $p\mid d_n$. Note that $\det W=ud_1d_2\cdots d_n$ for some unit $u$.  As $p^2\not\mid \det W$, we find that $p\not\mid d_k$ for any $k\in\{1,2,\ldots,n-1\}$ and $p^2\not\mid d_n$. Thus, $\rank_p W=n-1$. It follows from Lemmas \ref{ellz0} and  \ref{solutionmodp2} that  $W^*z\equiv 0\pmod{p^2}$ has a solution $z\not\equiv 0\pmod{p}$. Using Lemma \ref{psolutionanddn}, $p^2\mid d_n$. This is a contradiction and hence completes the proof. $\Box$

\subsection{The case $p=1+i$}
\begin{theorem}\label{maineven}
Let $G\in \mathcal{G}_n^{sc}$ such that $\det W(G)\neq 0$. Let $U\in \mathscr{U}_G$ with level $\ell$. If $2^{\lfloor\frac{n}{2}\rfloor+1}\not\mid \det W$, then $2\not\mid \ell$.
\end{theorem}
Set $k=\lfloor\frac{n}{2}\rfloor$. Let $\tilde{W}$ and $\tilde{W}_1$ be the matrix defined as follows:
\begin{equation}
\tilde{W}= \begin{cases}
  (e,Ae,\ldots,A^{k-1}e) &\mbox{if $n$ is even,}\\
    (Ae,A^2e,\ldots,A^{k}e) &\mbox{if $n$ is odd,}
   \end{cases}
\end{equation}
and
\begin{equation}
\tilde{W}_1= \begin{cases}
  (e,A^2e,\ldots,A^{2k-2}e) &\mbox{if $n$ is even,}\\
    (Ae,A^3e\ldots,A^{2k-1}e) &\mbox{if $n$ is odd.}
   \end{cases}
\end{equation}
\begin{lemma}\label{fundasol}
Let $G\in\mathcal{G}_n^{sc}$ such that $\frac{\det W}{2^{\lfloor\frac{n}{2}\rfloor}}$ is odd. Then the columns of $\tilde{W}$ constitute a set of fundamental solutions to $W^*z\equiv 0\pmod {1+i}.$
\end{lemma}
\begin{proof}
We only consider the case that $n$ is even while the odd case can be settled in a similar way. Note that $e^*e=n\equiv 0\pmod{2}$. By Lemma \ref{totalsumAk}, we finds that  $W^*\tilde{W}\equiv 0\pmod{2}$ and hence $W^*\tilde{W}\equiv 0\pmod {1+i}$. Thus, each columns of $\tilde{W}$ is a solution to $W^*z\equiv 0\pmod {1+i}$. By Corollary  \ref{rankceil}, $\rank_{1+i} W=\lceil\frac{n}{2}\rceil$ and hence any  set of fundamental solutions has exactly $n-\lceil\frac{n}{2}\rceil=k$ vectors. Note that $\tilde{W}$ has exactly $k$ columns. By Lemma \ref{firstr},  these $k$ columns are linearly independent and hence  constitute a set of fundamental solutions.
\end{proof}

By Lemma \ref{totalsumAk} and the fact that $e^*e\equiv 0\pmod{2}$ when $n$ is even, one easily sees that all entries in $W^*\tilde{W}_1$ are divisible by $2$. That is, $\frac{W^*\tilde{W}_1}{2}$ is Gaussian integral. We show that this matrix has full column rank over $\mathbb{Z}[i]/(1+i)$, which is a generalization of \cite[Lemma 3.10]{wang2017JCTB} for undirected graphs. The previous proof can be extended easily to mixed graphs.
\begin{lemma}\label{rankWW1}
Let $G\in\mathcal{G}_n^{sc}$ such that $\frac{\det W}{2^{\lfloor\frac{n}{2}\rfloor}}$ is odd. Then we have $\rank_{1+i} \frac{W^*\tilde{W}_1}{2}=\lfloor\frac{n}{2}\rfloor$.
\end{lemma}
 \begin{proof}
 We consider the following two cases:\\
 \noindent\emph{Case} 1: $n$ is even.

  By Theorem \ref{Wrealorimag}, we have
 $\det W=u2^{n/2}b$, where $u$ is a unit and $b$ is an odd integer. Thus, $\det W^*W=2^nb^2$ and hence $\det \frac{W^*W}{2}=b^2$. Therefore, $\rank_{1+i}\frac{W^*\tilde{W}}{2}=n$.  Thus, the $n$ columns of $\frac{W^*\tilde{W}}{2}$ are linearly independent, which clearly implies that $ \frac{W^*\tilde{W}_1}{2}$  are also linearly independent. Thus $\rank_{1+i} \frac{W^*\tilde{W}_1}{2}=\lfloor\frac{n}{2}\rfloor$.

 \noindent\emph{Case} 2: $n$ is odd.

 Let $W'$ be the matrix obtained from $W$ by doubling the first column. Then $\frac{W^*W'}{2}$ is Gaussian integral. As  $\det W=u2^{(n-1)/2}b$ (b is odd), we have
 $\det W^*W'=2^nb^2$ and hence $\frac{ W^*W'}{2}$ has full rank $n$.  Thus, $\frac{W^*\tilde{W}_1}{2}$ must have full column rank, i.e, $\rank_{1+i} \frac{W^*\tilde{W}_1}{2}=\lfloor\frac{n}{2}\rfloor$.
 \end{proof}

 \begin{lemma}\label{uandx}
 Let $G\in\mathcal{G}_n^{sc}$ such that $\frac{\det W}{2^{\lfloor\frac{n}{2}\rfloor}}$ is odd. Let $U\in \mathscr{U}_G$ has level $\ell$. If $2\mid \ell$ then there exists a Gaussian integral $n$-dimensional vector $u\not\equiv 0\pmod {1+i}$ and a $\lfloor\frac{n}{2}\rfloor$-dimensional vector $x\not\equiv 0\pmod{1+i}$ such that $u^*A^ku\equiv 0\pmod{4}$, $W^*u\equiv 0\pmod {2}$ and $\tilde{W} x\equiv u\pmod {2}$.
 \end{lemma}
 \begin{proof}
 Let $H\in\mathcal{G}_n^{sc}$ such that $\mathscr{U}_G(H)=\{U\}$. Write $U_1=\ell U$. Note that $1+i$ divides $\ell$. Due to the minimality of $N(\ell)$, $U_1$ contains a column $u$ such that $u\not\equiv 0\pmod{1+i}$. Since $U_1^*A^k(G)U_1=\overline{\ell}\ell U^*A^k(G)U=\overline{\ell}\ell A^k(H)$ and $2\mid \ell$, we have  $U_1^*A^kU_1\equiv 0\pmod{4}$ and hence $u^*A^ku\equiv 0\pmod 4$. Similarly, we have $W^*(G)U_1=\ell W^*(G)U=\ell W^*(H)\equiv 0\pmod {2}$ and hence $W^*u\equiv 0\pmod{2}$. It remains to find a vector $x\not\equiv 0\pmod{1+i}$ such that  $\tilde{W} x\equiv u\pmod{2}$.

Since $W^*u\equiv 0\pmod{2}$, we have $W^*u\equiv 0\pmod{1+i}$. It follow from Lemma \ref{fundasol} that $u$ can be expressed as a linear combination of columns of $\tilde{W}$, that is, there exists a vector $v$ such that
$u\equiv \tilde{W}v\pmod{1+i}$.

\noindent\textbf{Claim}: $\tilde{W}y\equiv \frac{u-\tilde{W}v}{1+i}\pmod{1+i}$ has a solution for unknown vector $y$.

Write $b=\frac{u-\tilde{W}v}{1+i}$. Let $z$ be any vector satisfying $\tilde{W}^*z\equiv 0\pmod{1+i}$. If we can show that $b^*z\equiv 0\pmod{1+i}$ always hold, then the equations $\tilde{W}^*x\equiv 0\pmod{1+i}$ and $(\tilde{W},b)^*x\equiv 0\pmod{1+i}$ have the same solutions, which implies that $\rank_{1+i}\tilde{W}^*=\rank_{1+i}(\tilde{W},b)^*$, i.e., $\rank_{1+i}\tilde{W}=\rank_{1+i}(\tilde{W},b)$ and hence the claim follows.

As $\tilde{W}^*$ has full column rank $k=\lfloor\frac{n}{2}\rfloor$,  the solution space of $\tilde{W}^*x\equiv 0\pmod{1+i}$ has dimension $\lceil\frac{n}{2}\rceil$. As $\tilde{W}^*W\equiv 0\pmod{2}$, we have $\tilde{W}^*W\equiv 0\pmod{1+i}$. By Corollary \ref{rankceil} we have $\rank_{1+i} W=\lceil\frac{n}{2}\rceil$ and hence $z$ belongs to the column space of  $W$. Thus, we can write $z\equiv Wa \pmod{1+i}$ for some vector $a$. Therefore,
  \begin{eqnarray}
z^*b&\equiv& (Wa)^*\frac{u-\tilde{W}v}{1+i}\\
&\equiv&a^*\frac{W^*u-W^*\tilde{W}v}{1+i}\\
&\equiv &0\pmod{1+i},
\end{eqnarray}
where the last congruence holds because  $W^*u\equiv 0\pmod{2}$ and $W^*\tilde{W}\equiv 0\pmod 2$. Thus,  $b^*z\equiv 0\pmod{1+i}$ and the Claim holds.

Let $x=v+(1+i)y$. By the claim, we have $\tilde{W}(1+i)y\equiv u-\tilde{W}v\pmod{2}$. Thus, $\tilde{W}x\equiv u\pmod{2}$. Finally, as $u\not\equiv 0\pmod{1+i}$, we must have $x\not\equiv 0\pmod{1+i}$. This completes the proof of this lemma.
 \end{proof}
\noindent\emph{Proof of Theorem \ref{maineven}.}
Suppose to the contrary that $2\mid \ell$. Let $u$ and $x$ be vectors described as in Lemma \ref{uandx}. As $\tilde{W}x\equiv u\pmod{2}$, we have $u=\tilde{W}x+2\beta$ for some vector $\beta$. It follows that
  \begin{eqnarray}\label{removebeta}
u^*A^ju&=&(\tilde{W}x+2\beta)^*A^j(\tilde{W}x+2\beta)\nonumber\\
       &=&x^*\tilde{W}^*A^j\tilde{W}x+2\beta^*A^j\tilde{W}x+2x^*\tilde{W}^*A^j\beta+4\beta^*A^j\beta\nonumber\\
       &=&x^*\tilde{W}^*A^j\tilde{W}x+4\Re(\beta^*A^j\tilde{W}x)+4\beta^*A^j\beta\nonumber\\
       &\equiv&x^*\tilde{W}^*A^j\tilde{W}x\pmod{4}\label{uAju}.
\end{eqnarray}
Since $u^*A^ju\equiv 0\pmod{4}$ by Lemma \ref{uandx}, from (\ref{uAju}), we have
\begin{equation}\label{xWAWx}
x^*\tilde{W}^*A^j\tilde{W}x\equiv 0\pmod{4}\quad\text{ for } j=0,1,\ldots,n-1.
\end{equation}
 Define
 \begin{equation}\label{defj}
 (j_1,j_2,\ldots,j_k)=\begin{cases}
 (0,1,\ldots,\frac{n}{2}-1)&\mbox{if $n$ is even,}\\
 (1,2,\ldots,\frac{n-1}{2})&\mbox{if $n$ is odd,}
 \end{cases}
 \end{equation}
where $k=\lfloor\frac{n}{2}\rfloor$. Then $\tilde{W}=[A^{j_1}e,A^{j_2}e,\ldots,A^{j_k}e]$. Let $R^{(l)}=(r_{s,t}^{(l)})=\tilde{W}^*A^l\tilde{W}$, $l=0,1,\ldots,n-1$. Then we have
 \begin{equation}\label{rstl}
 r_{s,t}^{(l)}=e^*A^{j_s+j_t+l}e.
 \end{equation}
  We claim that $R^{(l)}$ is a real symmetric matrix with each entry even. Clearly, by (\ref{rstl}), $r_{s,t}^{(l)}$ is real and $r_{s,t}^{(l)}=r_{t,s}^{(l)}$. It remains to show that $2\mid r_{s,t}^{(l)}$. If $j_s+j_t+l>0$ then the claim follows by Lemma \ref{totalsumAk}. Now assume $j_s=j_t=l=0$. According to (\ref{defj}), $n$ must be even in this case. Thus, $e^*A^{j_s+j_t+l}e$ is an even integer and the claim also holds. This proves the claim. It follows that
\begin{eqnarray*}
x^*\tilde{W}^*A^l\tilde{W}x&=&\sum_{1\le s,t\le k}x_s^*r_{s,t}^{(l)}x_t\\
&=&\sum_{1\le s\le k}x_s^*r_{s,s}^{(l)}x_s+\sum_{1\le s<t\le k}x_s^*r_{s,t}^{(l)}x_t+\sum_{1\le t<s\le k}x_s^*r_{s,t}^{(l)}x_t\\
&=&\sum_{1\le s\le k}x_s^*r_{s,s}^{(l)}x_s+\sum_{1\le s<t\le k}\big(x_s^*r_{s,t}^{(l)}x_t+x_t^*r_{t,s}^{(l)}x_s\big)\\
&=&\sum_{1\le s\le k}x_s^*r_{s,s}^{(l)}x_s+\sum_{1\le s<t\le k}r_{s,t}^{(l)}\big(x_s^*x_t+\overline{x_s^*x_t}\big)\\
&=&\sum_{1\le s\le k}x_s^*r_{s,s}^{(l)}x_s+\sum_{1\le s<t\le k}2r_{s,t}^{(l)}\Re(x_s^*x_t)\\
&\equiv &\sum_{1\le s\le k}x_s^*r_{s,s}^{(l)}x_s\pmod{4}.
\end{eqnarray*}
Thus, from (\ref{xWAWx}), we have
\begin{equation}\label{alpharalpha}
\sum_{1\le s\le k}x_s^*r_{s,s}^{(l)}x_s\equiv 0\pmod{4}.
\end{equation}
Note that $x_s^*x_s$ and $x_s$ have the same Gaussian parity, i.e., $x_s^*x_s\equiv x_s\pmod{1+i}$. As $r_{s,s}^{(l)}$ is an even integer, we have
$x_s^*r_{s,s}^{(l)}x_s\equiv r_{s,s}^{(l)}x_s\pmod{2(1+i)}$. Thus, from (\ref{alpharalpha}), we have
\begin{equation}\label{alpharalpha2}
[r_{1,1}^{(l)},r_{2,2}^{(l)},\ldots,r_{k,k}^{(l)}]x\equiv 0\pmod{2(1+i)},~ l=0,1,\ldots,n-1.
\end{equation}
Moreover, we have
\begin{eqnarray*}
&&\left(\begin{matrix}
r_{1,1}^{(0)}&r_{2,2}^{(0)}&\ldots&r_{k,k}^{(0)}\\
r_{1,1}^{(1)}&r_{2,2}^{(1)}&\ldots&r_{k,k}^{(1)}\\
\cdots&\cdots&\cdots&\cdots\\
r_{1,1}^{(n-1)}&r_{2,2}^{(n-1)}&\cdots&r_{k,k}^{(n-1)}\\
\end{matrix}
\right)\\
&=&\left(\begin{matrix}
e^*A^{2j_1}e&e^*A^{2j_2}e&\cdots&e^*A^{2j_k}e\\
e^*A^{2j_1+1}e&e^*A^{2j_2+1}e&\cdots&e^*A^{2j_k+1}e\\
\vdots&\vdots&\cdots&\vdots\\
e^*A^{2j_1+n-1}e&e^*A^{2j_2+n-1}e&\cdots&e^*A^{2j_k+n-1}e\\
\end{matrix}
\right)\\
&=&\left(
\begin{matrix}
e^*\\e^*A\\\vdots\\e^*A^{n-1}\\
\end{matrix}
\right)
\left(A^{2j_1}e,A^{2j_2}e,\ldots,A^{2j_k}e\right)\\
&=&W^*\tilde{W}_1.
\end{eqnarray*}
Thus, we can rewrite (\ref{alpharalpha2}) as
\begin{equation}
W^*\tilde{W}_1x\equiv 0\pmod{2(1+i)}.
\end{equation}
As $\frac{W^*\tilde{W}_1}{2}$ is Gaussian integral, the equation is equivalent to
\begin{equation}
\frac{W^*\tilde{W}_1}{2}x\equiv 0\pmod{1+i}.
\end{equation}
From Lemma \ref{rankWW1}, we know that $\rank_{1+i}\frac{\tilde{W}^*\tilde{W}_1}{2}=\lfloor\frac{n}{2}\rfloor$.   Thus, $x\equiv 0\pmod{1+i}$. This contradiction completes the proof of  Theorem \ref{maineven}. $\Box$

\noindent\emph{Proof of Theorem \ref{main}.}  Theorem \ref{main} follows immediately from Theorems \ref{mainodd} and \ref{maineven}. $\Box$
\section{More discussions on $p=1+i$}
By Theorem \ref{main},  Conjecture \ref{mainconj} (or  Conjecture \ref{mainconj2}) can be reduced to the following
\begin{conjecture}\label{mainconjsim}
Let $G\in \mathcal{G}_n^{sc}$. If $\frac{\det W}{2^{\lfloor n/2 \rfloor}}$ is square-free in  $\mathbb{Z}[i]$, then for any $U\in \mathscr{U}_G$, $\ell(U)\neq 1+i$.
\end{conjecture}
We shall show Conjecture \ref{mainconjsim} is true for the special case that $G$ is an ordinary graph (with no directed edges). This result strengthens Theorem \ref{mainforgraph}. That is, even in $\mathcal{G}_n^{sc}$, which includes all undirected graphs on $[n]$ as a proper subset, the only graphs $\mathbb{R}$-cospectral to $G$ are isomorphic to $G$ if $G$ satisfies the condition of Theorem \ref{mainforgraph}.

For a positive integer $n$, let $\mathcal{M}_n$ denote the set of all $n\times n$ Hermitian Gaussian integral matrices with vanishing diagonal entries.
For  $k\ge 1$ and $s\ge 0$, let $$U_{k,s}=\left(\begin{matrix}
U_0&&&&\\
&U_0&&&\\
&&\ddots&&\\
&&&U_0&\\
&&&&I_s\\
\end{matrix}
\right)
$$
be a matrix of order $2k+s$, where $U_0=\frac{1}{1+i}\left(\begin{matrix}
1&i\\
i&1\\
\end{matrix}
\right)
$ and $I_s$ is the identity matrix of order $s$. It is easy to see that $U^*U=I$, $Ue=e$ and $\ell(U_{k,s})=1+i$.
\begin{lemma}\label{stdform}
Let $U$ be a $n\times n$ Gaussian rational unitary matrix with $Ue=e$ and  $\ell(U)=1+i$. Then there exist two permutation matrices $P$ and $Q$ such that
$PUQ=U_{k,n-2k}$ for some $k\ge 1$.
\end{lemma}
\begin{proof}
We prove the lemma by induction on $n$. Since $\ell(U)=1+i$ and $Ue=e$, $U$ has a row which contains at least two nonzero entries. Let $\tilde{U}=(1+i)U$ and $P_1,Q_1$ be two permutation matrix such that the first two entries of the first row in $P_1\tilde{U}Q_1$ are non-zero. Let $(a_1,a_2,\ldots,a_n)$ denote the first row of $P_1\tilde{U}Q_1$. Now we have $a_1+a_2+\cdots+a_n=1+i$ and $|a_1|^2+|a_2|^2+\cdots+|a_n|^2=2$. As each $a_j$ is Gaussian integral and $a_1,a_2\neq 0$, one must have $(a_1,a_2,\ldots,a_n)=(1,i,0,\ldots,0)$ or $(i,1,0,\ldots,0)$. We may assume $(a_1,a_2,\ldots,a_n)=(1,i,0,\ldots,0)$ since otherwise we can interchange the first two columns of $Q_1$. Let $\alpha$ and $\beta$ denote the first and second columns of $P_1\tilde{U}Q_1$. Note that $U^\T e=e$. Similar considerations indicate that both $\alpha$ and $\beta$ have exactly two non-zeros entries (1 and $i$), thus $\alpha=(1,0,\ldots,0,i,0,\ldots,0)$ and $\beta=(i,0,\ldots,0,1,0,\ldots,0)$. As $\alpha^*\beta=0$ the position of $i$ in $\alpha$ agrees with the position of $1$ in $\beta$. Thus, there exists a permutation matrix $P_2$ such that $P_2P_1\tilde{U}Q_1$ has the following form.
\begin{equation}P_2P_1\tilde{U}Q_1=\left(\begin{matrix}
1&i&0&\cdots&0\\
i&1&0&\cdots&0\\
0&0&*&\cdots&*\\
\vdots&\vdots&\vdots&\cdots&\vdots\\
0&0&*&\cdots&*\\
\end{matrix}
\right).
\end{equation}
Equivalently,
\begin{equation}P_2P_1UQ_1=\left(\begin{matrix}
U_0&&\\
&&U_1\\
\end{matrix}
\right)
\end{equation}
for some $U_1$ of order $n-2$.

If $n=2$ we are done. Suppose that $n\ge 3$.  Let $e_1$ denote the all-one vector of dimension $n-2$. It is easy to see that $U_1$ is a Gaussian rational unitary matrix with $U_{1}e_1=e_1$. Moreover, $\ell(U_1)|\ell(U)$, that is  $\ell(U_1)\in \{1,1+i\}$. If  $\ell(U_1)=1$ then $U_1$ is a permutation matrix. Let
\begin{equation}Q_2=\left(\begin{matrix}
I_2&&\\
&&U_1^{-1}\\
\end{matrix}
\right).
\end{equation}
Then $Q_2$ is a permutation matrix and  $P_2P_1UQ_1Q_2=U_{1,n-2}$. This proves the lemma for the case that $\ell(U_1)=1$.  If $\ell(U_1)=1+i$, by induction hypothesis, there exist  permutation matrices $P'$ and $Q'$ such that
$P'U_1Q'=U_{k',n-2-2k'}$. Let
\begin{equation}P_3=\left(\begin{matrix}
I_2&&\\
&&P'\\
\end{matrix}
\right)
\quad\text{and}\quad
Q_3=\left(\begin{matrix}
I_2&&\\
&&Q'\\
\end{matrix}
\right).
\end{equation}
Then $P_3P_2P_1UQ_1Q_3=U_{1+k',n-2-2k'}.$ This completes the proof of this lemma.
\end{proof}
\begin{lemma}\label{AEB}
Let $A$ be a (0,1)-matrix of order $n$ and $B=U_{k,s}^*AU_{k,s}$ where $2k+s=n$. If each entry of $B$ belongs to $\{0,1,i,-i\}$, then $A=B$.
\end{lemma}
\begin{proof}
Write \begin{equation}
A=\left(\begin{matrix}
A_{1,1}&A_{1,2}&\cdots&A_{1,k}&A_{1,k+1}\\
A_{2,1}&A_{2,2}&\cdots&A_{2,k}&A_{2,k+1}\\
\cdots&\cdots&\cdots&\cdots&\cdots\\
A_{k,1}&A_{k,2}&\cdots&A_{k,k}&A_{k,k+1}\\
A_{k+1,1}&A_{k+1,2}&\cdots&A_{k+1,k}&A_{k+1,k+1}\\
\end{matrix}
\right),
\end{equation}
where $A_{j,j}$ is a square matrix of order 2 for $j\in[k]$, and $A_{k+1,k+1}$ is of order $s$. We have
\begin{equation}
U_{k,s}^*AU_{k,s}=\left(\begin{matrix}
U_0^*A_{1,1}U_0&U_0^*A_{1,2}U_0&\cdots&U_0^*A_{1,k}U_0&U_0^*A_{1,k+1}\\
U_0^*A_{2,1}U_0&U_0^*A_{2,2}U_0&\cdots&U_0^*A_{2,k}U_0&U_0^*A_{2,k+1}\\
\cdots&\cdots&\cdots&\cdots&\cdots\\
U_0^*A_{k,1}U_0&U_0^*A_{k,2}U_0&\cdots&U_0^*A_{k,k}U_0&U_0^*A_{k,k+1}\\
A_{k+1,1}U_0&A_{k+1,2}U_0&\cdots&A_{k+1,k}U_0&A_{k+1,k+1}\\
\end{matrix}
\right).
\end{equation}
Let $\Omega_1$ denote the set of all (0,1)-matrices $C$ of order 2 such that each entry of $U_0^*CU_0$ belongs to $\{0,1,i,-i\}$. Direct calculation shows that
$$\Omega_1=\left\{\left(\begin{matrix}0&0\\0&0\\\end{matrix}\right),\left(\begin{matrix}1&0\\0&1\\\end{matrix}\right),\left(\begin{matrix}0&1\\1&0\\\end{matrix}\right),\left(\begin{matrix}1&1\\1&1\\\end{matrix}\right)\right\}$$ and moreover, $U_0^*CU_0=C$ for each $C\in \Omega_1$. This proves that $U_0^*A_{i,j}U_0=A_{i,j}$ for $i,j\in[k]$. Similar argument shows that each column (resp. row) of $A_{i,k+1}$ (resp. $A_{k+1,i}$) is either all-zero or all-one. Therefore, $U_0^*A_{i,k+1}=A_{i,k+1}$ and $A_{k+1,i}U_0=A_{k+1,i}$. This completes the proof.
\end{proof}
The following corollary verifies Conjecture \ref{mainconjsim} for the special case that $G$ contains no directed edges.
\begin{corollary}
Let $G$ be an undirected graph satisfying the conditions of Theorem~\ref{main}. Then for any $U\in \mathscr{U}_G$, $\ell(U)\neq 1+i$.
\end{corollary}
\begin{proof}
We prove the corollary by contradiction. Suppose  $\ell(U)= 1+i$  and $\{U\}=\mathscr{U}_G(H)$  for some $H\in \mathcal{G}_n^{sc}$. Then
$U^*A(G)U=A(H)$.

By Lemma \ref{stdform}, there exist two permutation matrices $P$ and $Q$ such that $PUQ=U_{k,n-2k}$, i.e., $U=P^*U_{k,n-2k}Q^*$ for some $k\ge 1$. Therefore, we have

\begin{equation}\label{PAQ}
(P^*U_{k,n-2k}Q^*)^*A(G)P^*U_{k,n-2k}Q^*=A(H).
\end{equation}
Write $A_1=PA(G)P^*$, $B_1=Q^*A(H)Q$ and let $G_1,H_1$ be two graphs with adjacency matrices $A_1$ and $B_1$ respectively. Now (\ref{PAQ}) is equivalent to
\begin{equation}\label{PAQsim}
U_{k,n-2k}^*A_1U_{k,n-2k}=B_1.
\end{equation}
It follows from Lemma \ref{AEB} that $A_1=B_1$, i.e., $G_1=H_1$. Now, from (\ref{PAQsim}), we have
 \begin{equation}\label{UW}
 U_{k,n-2k}^*W(G_1)=W(H_1)=W(G_1).
 \end{equation}
  As $W(G_1)=PW(G)$ and $\det W(G)\neq 0$, we have $\det W(G_1)\neq 0$, that is $W(G_1)$ is invertible. Note that $U_{k,n-2k}^*$ is not the identity matrix. This contradicts (\ref{UW}) and hence completes the proof.
  \end{proof}
  \section{Concluding remarks}
  In this paper, we are mainly concerned with the generalized spectral characterizations of self-converse mixed graphs. Given a self-converse mixed graph $G$ of order $n$ such that $\frac{\det W}{2^{\lfloor\frac{n}{2}\rfloor}}$ (which is always a real or pure imaginary integer) is square-free in $\mathbb{Z}[i]$, we showed that for any self-converse mixed graph $H$ that is $\mathbb{R}$-cospectral to $G$, there exists a Gaussian rational unitary matrix $U$ such that $Ue=e$, $U^*A(G)U=A(H)$ and $(1+i)U$ is a Gaussian integral matrix. Such a unitary matrix $U$ is very close to a permutation matrix, and therefore gives strong evidences for the conjecture that self-converse mixed graphs satisfying the above condition are DGS. Our main result also implies that for an ordinary graph $G$ (viewed as a mixed graph) satisfying the above property, any self-converse mixed graph $H$ that is $\mathbb{R}$-cospectral with $G$ is isomorphic to $G$. This strengthens a recent result of the first author~\cite{wang2017JCTB}.
  However, regarding Conjecture 7, new insights and techniques are still needed to eliminate the possibility that $\ell(U)=1+i$. We leave it as an interesting and challenging future work.

%section*{Acknowledgments}
%The authors thank the referees for their careful reading and valuable suggestions.
 
\end{document}